\documentclass[11pt,reqno]{amsart}

\usepackage{amssymb, amsmath, amsthm}
\usepackage{hyperref}
\usepackage[alphabetic,lite]{amsrefs}
\usepackage{verbatim}
\usepackage{amscd}   
\usepackage[all]{xy} 
\usepackage{youngtab} 
\usepackage{young} 
\usepackage{ytableau}
\usepackage{tikz}
\usetikzlibrary{arrows}
\usepackage{ mathrsfs }
\usepackage{cases}
\usepackage{array}
\usepackage{tabu}
\usepackage{calligra,mathrsfs}
\usepackage{dsfont}
\usepackage{upgreek}

\newcommand{\defi}[1]{{\upshape\sffamily #1}}
\DeclareMathOperator{\ShHom}{\mathscr{H}\text{\kern -3pt {\calligra\large om}}\,}

\newcommand{\D}{\mathcal{D}}

\newcommand{\pth}{\operatorname{path}}

\newcommand{\cpc}{\operatorname{cap}}

\newcommand{\Eu}{\operatorname{Eu}}

\newcommand{\opmod}{\operatorname{mod}}
\newcommand{\op}{\operatorname}

\newcommand{\bb}[1]{\mathbb{#1}}

\newcommand{\ol}[1]{\overline{#1}}

\newcommand{\ul}[1]{\underline{#1}}

\newcommand{\redit}[1]{\textcolor{red}{#1}}

\newtheorem{theorem}{Theorem}[section]
\newtheorem*{theorem*}{Theorem}
\newtheorem*{problem*}{Problem}
\newtheorem{lemma}[theorem]{Lemma}

\newtheorem*{corollary*}{Corollary}

\newtheorem*{main-thm*}{Main Theorem}
\newtheorem*{Euler-obstructions*}{Theorem on Euler obstructions}
\newtheorem*{char-cycles*}{Theorem on characteristic cycles of simple $\D_X$-modules}
\newtheorem*{local-Euler*}{Theorem on intersection cohomology local Euler characteristics}
\newtheorem*{deRham*}{Theorem on de Rham cohomology}

\theoremstyle{definition}
\newtheorem{definition}[theorem]{Definition}
\newtheorem*{definition*}{Definition}
\newtheorem{example}[theorem]{Example}

\theoremstyle{remark}
\newtheorem{remark}[theorem]{Remark}
\newtheorem*{remark*}{Remark}

\numberwithin{equation}{section}

\tikzset{
  treenode/.style = {align=center, inner sep=0pt, text centered,solid,thin,
    font=\sffamily},
  arn_n/.style = {treenode, circle, white, font=\sffamily\bfseries, draw=black,
    fill=black, text width=.5em},
  arn_nl/.style = {treenode, circle, white, font=\sffamily\bfseries, draw=black,
    fill=black, text width=2.5em},  
  arn_r/.style = {treenode, circle, red, draw=red, 
    text width=.5em, very thick},
  arn_v/.style = {treenode, circle, black, font=\sffamily\bfseries, draw=black, text width=1.2em},
  arn_x/.style = {treenode, rectangle, draw=black,
    minimum width=.5em, minimum height=0.5em},
  dott/.style={edge from parent/.style={dotted, very thick,circle,draw}},
  emph/.style={edge from parent/.style={dashed, very thick,circle,draw}},
  norm/.style={edge from parent/.style={solid,thin,circle,draw}}
}

\begin{document}

\title{Euler obstructions for the Lagrangian Grassmannian}

\author{Paul LeVan}
\address{Department of Mathematics, University of Notre Dame, 255 Hurley, Notre Dame, IN 46556}
\email{plevan@nd.edu}

\author{Claudiu Raicu}
\address{Department of Mathematics, University of Notre Dame, 255 Hurley, Notre Dame, IN 46556\newline
\indent Institute of Mathematics ``Simion Stoilow'' of the Romanian Academy}
\email{craicu@nd.edu}


\date{\today}


\begin{abstract} 
 We prove a case of a positivity conjecture of Mihalcea--Singh, concerned with the local Euler obstructions associated to the Schubert stratification of the Lagrangian Grassmannian $LG(n,2n)$. Combined with work of Aluffi--Mihalcea--Sch\"urmann--Su, this further implies the positivity of the Mather classes for Schubert varieties in $LG(n,2n)$, which Mihalcea--Singh had verified for the other cominuscule spaces of classical Lie type. Building on the work of Boe and Fu, we give a positive recursion for the local Euler obstructions, and use it to show that they provide a positive count of admissible labelings of certain trees, analogous to the ones describing Kazhdan-Lusztig polynomials. Unlike in the case of the Grassmannians in types A and D, for $LG(n,2n)$ the Euler obstructions $e_{y,w}$ may vanish for certain pairs $(y,w)$ with $y\leq w$ in the Bruhat order. Our combinatorial description allows us to classify all the pairs $(y,w)$ for which $e_{y,w}=0$. Restricting to the big opposite cell in $LG(n,2n)$, which is naturally identified with the space of $n\times n$ symmetric matrices, we recover the formulas for the local Euler obstructions associated with the matrix rank stratification.
\end{abstract}

\maketitle

\section{Introduction}\label{sec:intro}

The goal of this note is to study the local Euler obstructions associated with the Schubert stratification of the Lagrangian Grassmannian $X=LG(n,2n)$. It was conjectured by Mihalcea--Singh \cite{mihalcea-singh}*{Conjecture~10.2} that these invariants are non-negative, in the more general context when $X$ is a \defi{cominuscule space} (also referred to as a \defi{compact Hermitian symmetric space} in the literature). The conjecture was verified in \cite{mihalcea-singh}*{Theorem~10.4} for the cominuscule spaces of classical Lie type other than $LG(n,2n)$, mainly as a consequence of the work of Boe and Fu \cite{boe-fu}, and of Bressler--Finkelberg--Lunts \cite{BFL} for the type A Grassmannian. The case $X=LG(n,2n)$ is also considered in \cite{boe-fu}, where a recursive and a combinatorial description of the local Euler obstructions is given, both of which are based on formulas with positive and negative contributions. Our input is to explain how to make these formulas positive, and in particular to confirm the Mihalcea--Singh conjecture for $LG(n,2n)$. We are moreover able to completely classify when the Euler obstructions vanish, and to recover the explicit formulas from \cite{zhang}*{Theorem~6.6} and \cite{lor-rai}*{Corollary~5.3} for the corresponding invariants associated to the rank stratification on the space of $n\times n$ symmetric matrices.

Following \cite{boe}*{Section~3}, we index the Schubert cells by words $w$ of length $n$ in the alphabet $\{\alpha,\beta\}$, write $X_w$ for the cell corresponding to $w$, and $\ol{X}_w$ for its closure.  Every word $w$ is uniquely represented by a \defi{path} in the $2$-plane starting at the origin, where each $\alpha$ represents a line segment from $(a,b)$ to $(a+1,b-1)$, and $\beta$ represents a line segment from $(a,b)$ to $(a+1,b+1)$. We denote the path associated to $w$ by $\pth(w)$, and we have for instance
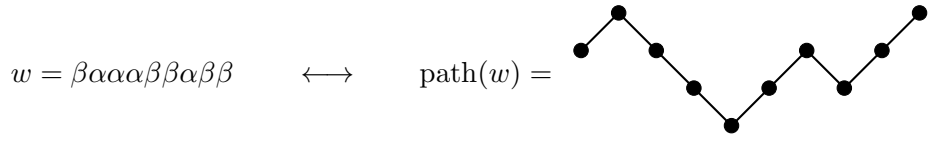
\[
w = \beta\alpha\alpha\alpha\beta\beta\alpha\beta\beta
\qquad\longleftrightarrow\qquad
\pth(w)=
\begin{array}{c}
\begin{tikzpicture}[scale=0.5]
\draw [thick] (0,0)--(1,1)--(2,0)--(3,-1)--(4,-2)--(5,-1)--(6,0)--(7,-1)--(8,0)--(9,1);
\node[arn_n] at (0,0) {};
\node[arn_n] at (1,1) {};
\node[arn_n] at (2,0) {};
\node[arn_n] at (3,-1) {};
\node[arn_n] at (4,-2) {};
\node[arn_n] at (5,-1) {};
\node[arn_n] at (6,0) {};
\node[arn_n] at (7,-1) {};
\node[arn_n] at (8,0) {};
\node[arn_n] at (9,1) {};
\end{tikzpicture}
\end{array}
\]
The \defi{Bruhat order}, given by $y\leq w$ if $X_y \subseteq\ol{X}_w$, can be interpreted pictorially by the fact that no point on $\pth(y)$ lies strictly above $\pth(w)$. We define the \emph{local Euler obstructions coefficients}
\[ e_{y,w} = \Eu_{\ol{X}_w}(p)\text{ for }p\in X_y,\]
where $\Eu_V$ denotes MacPherson's local Euler obstruction function associated to a subvariety $V$ \cite{macpherson}*{Section~3}; see also \cite{lor-rai} for a leisurely treatment, with extensive references, of the theory of Euler obstructions in the closely related case of matrix spaces. We write $|w|$ for the \defi{length} of the word $w$, and whenever we write $e_{y,w}$ we implicitly assume that $|y|=|w|$. We will also write $|w|_{\alpha}$ for the number of $\alpha$'s in $w$, and define $|w|_{\beta}$ similarly. Our first result is the following positive recursion for the local Euler obstructions (see also \cite{boe-fu}*{Section~6} for a set of recursive relations involving both positive and negative contributions).

\begin{theorem}\label{thm:euler-recursion}
The local Euler obstructions are uniquely determined by the following recursive relations.
\begin{enumerate}
    \item $e_{w,w}=1$ and $e_{y,w}=0$ if $y\not\leq w$.
    \item If $y=y'\alpha$ and $w=w'\beta$ then $e_{y,w}=e_{y'\beta,w}$.
    \item If $y=y'\alpha\beta y''$ and $w=w'\sigma\tau w''$ with $|y'|=|w'|$ and $\sigma\tau\in\{\alpha\alpha,\beta\alpha,\beta\beta\}$ then $e_{y,w}=e_{y'\beta\alpha y'',w}$.
    \item If $y=y'\alpha\beta y''$ and $w=w'\alpha\beta w''$ with $|y'|=|w'|$ then $e_{y,w}=e_{y'\beta\alpha y'',w}+e_{y'y'',w'w''}$.
    \item If $y=y'\alpha\alpha$ and $w=w'\beta\alpha$ then $e_{y,w}=e_{y'\beta\beta,w}$.
    \item If $y=y'\alpha\alpha$ and $w=w'\alpha\alpha$ then $e_{y,w}=e_{y'\beta\beta,w}+e_{y',w'}$.
    \end{enumerate}
In particular, $e_{y,w}\geq 0$ for all $y,w$.
\end{theorem}

The non-negativity of Euler obstructions in Theorem~\ref{thm:euler-recursion} was conjectured in \cite{mihalcea-singh}*{Conjecture~10.2} in the more general setting of cominuscule spaces, and was verified in \cite{mihalcea-singh}*{Theorem~10.4} for the classical Lie types A, B and D. Mihalcea and Singh also obtained significant computational evidence in support of the conjecture for $LG(n,2n)$ (type C) and for the Cayley plane (type $E_6$). One of their motivations was to establish positivity properties for Mather classes \cite{mihalcea-singh}*{Conjecture~1.2(a)}. Using the positivity of the Schubert expansion of Chern--Schwartz--MacPherson classes of Schubert cells \cite{AMSS}*{Corollary~1.4}, the result for Mather classes can be deduced from the positivity of local Euler obstructions \cite{mihalcea-singh}*{Proposition 10.3(a)}. In particular, it follows from Theorem~\ref{thm:euler-recursion} that \cite{mihalcea-singh}*{Conjecture~1.2(a)} also holds for $LG(n,2n)$.

We illustrate Theorem~\ref{thm:euler-recursion} with some examples (which can be checked against \cite{mihalcea-singh}*{Table 3}).

\begin{example}\label{ex:worked-recurrence}
We consider \(y = \alpha \alpha \alpha \beta\), \(w = \beta \beta \alpha \beta\) and denote for clarity \( (y,w) := e_{y,w}\). To help clarify the use of the recursive formulas, we underline the subword in \(y\) to be changed and the corresponding position in \(w\), and we mark each resulting change in \(y\) in red:
\[
\renewcommand*{\arraystretch}{1.5}
\begin{array}{lclcl}
(\alpha \alpha \ul{\alpha \beta}, \beta \beta \ul{\alpha \beta}) & \overset{(4)}{=} & (\alpha \ul{\alpha \redit{\beta}} \redit{\alpha}, \beta \ul{\beta \alpha} \beta) + (\alpha \ul{\alpha}, \beta \ul{\beta}) & & \\

& \overset{(3),(2)}{=} & (\ul{\alpha \redit{\beta}} \redit{\alpha} \alpha, \ul{\beta \beta} \alpha \beta) + (\ul{\alpha \redit{\beta}}, \ul{\beta \beta}) & & \\

& \overset{(3),(3)}{=} & (\redit{\beta \alpha} \alpha \ul{\alpha}, \beta \beta \alpha \ul{\beta}) + (\redit{\beta} \ul{\redit{\alpha}}, \beta \ul{\beta}) & & \\

& \overset{(2),(2)}{=} & (\beta \alpha \ul{\alpha \redit{\beta}}, \beta \beta \ul{\alpha \beta}) + (\beta \redit{\beta}, \beta \beta) & & \\

& \overset{(4),(1)}{=} & \left[(\beta \ul{\alpha \redit{\beta}} \redit{\alpha}, \beta \ul{\beta \alpha} \beta) + (\beta \ul{\alpha}, \beta \ul{\beta})\right] + 1 & & \\

& \overset{(3),(2)}{=} & (\beta \redit{\beta \alpha} \ul{\alpha}, \beta \beta \alpha \ul{\beta}) + (\beta \redit{\beta}, \beta \beta) + 1 & & \\

& \overset{(2),(1)}{=} & (\beta \beta \alpha \redit{\beta}, \beta \beta \alpha \beta) + 2 = 3 & & \\
\end{array}
\]
If we keep $y$ the same and take $w=\beta\beta\alpha\alpha$, then (with some steps omitted) we have that
\[(\alpha \alpha \alpha \beta, \beta \beta \alpha \alpha) \overset{(3)}{=}
(\beta \alpha \alpha \alpha , \beta \beta \alpha \alpha) \overset{(6)}{=}
(\beta \alpha \beta \beta , \beta \beta \alpha \alpha) + (\beta \alpha , \beta \beta) \overset{(1),(2)}{=} 0 + (\beta \beta , \beta \beta) \overset{(1)}{=} 1,
\]
and if we take $y=\alpha\alpha\alpha\alpha$, $w=\beta\beta\beta\alpha$, then
\[(\alpha \alpha \alpha \alpha, \beta \beta \beta \alpha) \overset{(5)}{=}
(\alpha \alpha \beta \beta , \beta \beta \beta \alpha) \overset{(3)}{=}
(\beta \beta \alpha \alpha , \beta \beta \beta \alpha) \overset{(5)}{=}  (\beta \beta \beta \beta , \beta \beta \beta \alpha) \overset{(1)}{=} 0.
\]
\end{example}

Even though Theorem~\ref{thm:euler-recursion} completely settles the non-negativity of Euler obstructions, it is difficult to use in practice, for instance in order to understand when $e_{y,w}=0$. Our next goal is then to obtain a non-recursive, combinatorial description of the obstructions $e_{y,w}$. To do so, based on the standard combinatorial models from \cites{las-sch,boe,boe-fu} we associate to every word $w$ a rooted tree $A(w)$, and to every pair $(y,w)$ a \defi{diagram} (\defi{decorated tree}) $A(w/y)$, consisting of the tree $A(w)$ together with some additional data at each leaf, called the \defi{capacity} of the corresponding terminal edge. We refer the reader to Section~\ref{sec:trees} for the details of the construction and terminology, but we show here an example to give a flavor of the combinatorics involved.

\begin{example}\label{ex:A-w-y}
 Let $w=\beta\beta\alpha\beta\alpha\alpha\alpha\beta\alpha\alpha\beta\alpha\beta\beta\alpha$, and $y=\alpha\beta\alpha\alpha\alpha\alpha\alpha\alpha\beta\alpha\alpha\alpha\alpha\alpha\beta$. The paths of $y$ and $w$, as well as the corresponding capacities, and the decorated tree $A(w/y)$ are pictured below.
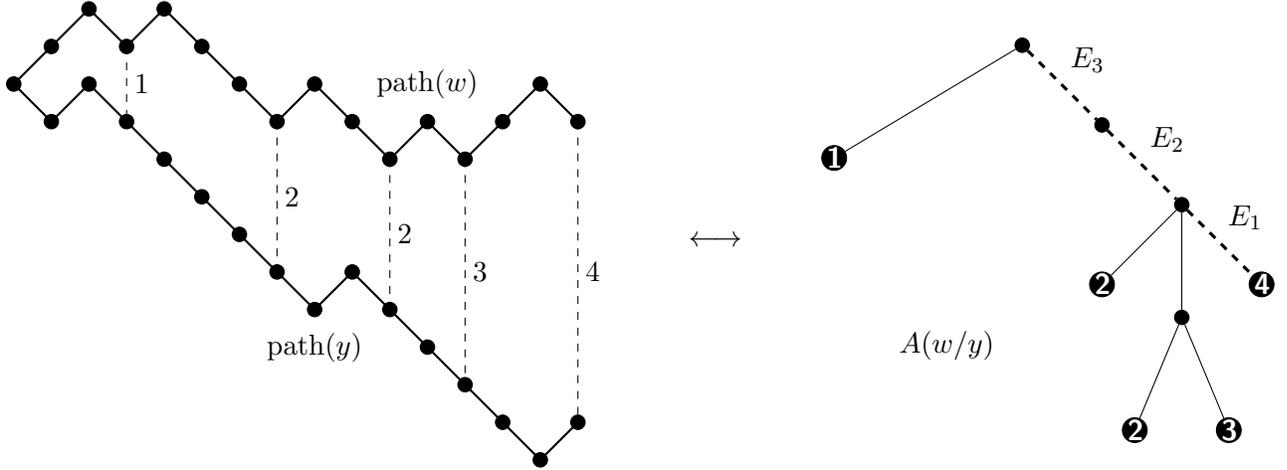
\[
\begin{array}{c}
\begin{tikzpicture}[scale=0.5]
\draw [thick] (0,0)--(1,1)--(2,2)--(3,1)--(4,2)--(5,1)--(6,0)--(7,-1)--(8,0)--(9,-1)--(10,-2)--(11,-1)--(12,-2)--(13,-1)--(14,0)--(15,-1);
\node[arn_n] at (0,0) {};
\node[arn_n] at (1,1) {};
\node[arn_n] at (2,2) {};
\node[arn_n] at (3,1) {};
\node[arn_n] at (4,2) {};
\node[arn_n] at (5,1) {};
\node[arn_n] at (6,0) {};
\node[arn_n] at (7,-1) {};
\node[arn_n] at (8,0) {};
\node[arn_n] at (9,-1) {};
\node[arn_n] at (10,-2) {};
\node[arn_n] at (11,-1) {};
\node[arn_n] at (12,-2) {};
\node[arn_n] at (13,-1) {};
\node[arn_n] at (14,0) {};
\node[arn_n] at (15,-1) {};
\draw [thick] (0,0)--(1,-1)--(2,0)--(3,-1)--(4,-2)--(5,-3)--(6,-4)--(7,-5)--(8,-6)--(9,-5)--(10,-6)--(11,-7)--(12,-8)--(13,-9)--(14,-10)--(15,-9);
\node[arn_n] at (1,-1) {};
\node[arn_n] at (2,0) {};
\node[arn_n] at (3,-1) {};
\node[arn_n] at (4,-2) {};
\node[arn_n] at (5,-3) {};
\node[arn_n] at (6,-4) {};
\node[arn_n] at (7,-5) {};
\node[arn_n] at (8,-6) {};
\node[arn_n] at (9,-5) {};
\node[arn_n] at (10,-6) {};
\node[arn_n] at (11,-7) {};
\node[arn_n] at (12,-8) {};
\node[arn_n] at (13,-9) {};
\node[arn_n] at (14,-10) {};
\node[arn_n] at (15,-9) {};
\draw [dashed] (3,1)--(3,-1); 
\node at (3.4,0) {$1$};
\draw [dashed] (7,-1)--(7,-5);
\node at (7.4,-3) {$2$};
\node at (8,-7) {$\pth(y)$};
\draw [dashed] (10,-2)--(10,-6);
\node at (10.4,-4) {$2$};
\node at (11,0) {$\pth(w)$};
\draw [dashed] (12,-2)--(12,-8);
\node at (12.4,-5) {$3$};
\draw [dashed] (15,-1)--(15,-9);
\node at (15.4,-5) {$4$};
\end{tikzpicture}
\end{array}
\qquad\longleftrightarrow\qquad
\begin{array}{c}
\begin{tikzpicture}[-,>=stealth',level/.style={sibling distance = 5cm/#1,level distance = 1.5cm}] 
\node [arn_n] {}
child{ node [arn_n] {1}
}
child[grow=south east,emph]{ node [arn_n] {}
  child[grow=south east,emph]{node [arn_n] {}
    child[grow=south west,norm]{ node [arn_n] {2} 
    }
    child[grow=south,norm]{ node [arn_n] {}
            child{ node [arn_n] {2} 
            }
            child{ node [arn_n] {3}
            }
		}
    child[grow=south east,emph]{ node [arn_n] {4}
  edge from parent node[above right] {$E_1$}
    }
  edge from parent node[above right] {$E_2$}
  }
  edge from parent node[above right] {$E_3$}
};
\node at (-1,-4) {$A(w/y)$};
\end{tikzpicture}
\end{array}
\]
The definition of capacities is given in \eqref{eq:def-cap-regular}, \eqref{eq:def-cap-distinguished}, the dotted edges $E_1,E_2,E_3$ are called \defi{distinguished edges}, while the remaining ones are \defi{regular edges}. The \defi{terminal edges} are the ones incident to leaves of the tree, and we write $\cpc(T)$ for the capacity of a terminal edge $T$.
\end{example}

We can now define \defi{admissible labelings}, which are essential to our non-recursive description of the local Euler obstructions. For edges $F,F'$, we write $F\leq F'$ if $F$ belongs to the shortest path joining $F'$ to the root.

\begin{definition}\label{def:admissible-label}
 We write $E_1,\cdots,E_r$ for the distinguished edges in $A(w/y)$. An \defi{admissible labeling} of $A(w/y)$ is a function that assigns to each edge $F$ a non-negative integer $\ell(F)$, satisfying the following properties:
 \begin{enumerate}
     \item If $F\leq F'$ then $\ell(F)\leq \ell(F')$.
     \item If $F$ is a regular terminal edge, then $\ell(F)\leq\cpc(F)$.
     \item\label{it:lE1-z0-empty} If $E_1$ is a terminal edge then $\ell(E_1)=\cpc(E_1)$.
     \item For every odd distinguished edge $E_{2i-1}$, $i\geq 1$, we have
     \[\ell(E_{2i-1}) = \min\{\ell(F) : F > E_{2i-1} \},\]
     with the exception when $i=1$ and $E_1$ is terminal, in which case $\ell(E_1)$ was given in \eqref{it:lE1-z0-empty}.
     \item For every $i\geq 1$ we have
     \[\ell(E_{2i-1}) \equiv \ell(E_{2i}) \ (\opmod 2),\]
     where we make the convention that $\ell(E_j)=0$ for $j>r$ (hence $\ell(E_r)$ is even if $r$ is odd).
 \end{enumerate}
If $A(w/y)$ has no edges then the empty labeling is the unique admissible labeling of $A(w/y)$.
\end{definition}

We will show that the diagram $A(w/y)$ in Example~\ref{ex:A-w-y} has $124$ admissible labelings (see Example~\ref{ex:124-labelings}), but for now we illustrate Definition~\ref{def:admissible-label} with two examples.

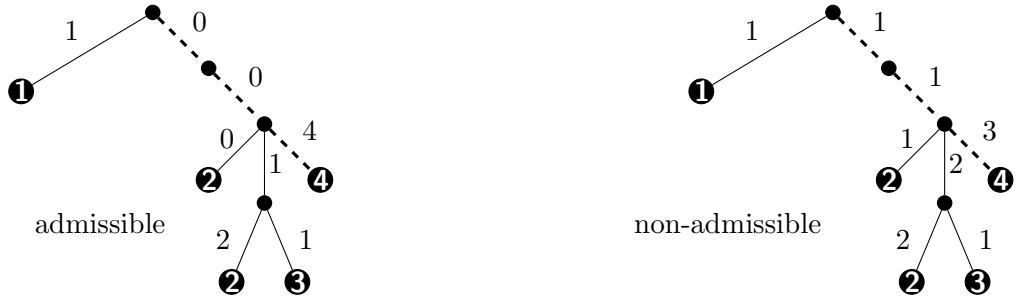
\[\begin{tikzpicture}[-,>=stealth',level/.style={sibling distance = 5cm/#1,level distance = 1.5cm},scale=0.7] 
\node [arn_n] {}
child{ node [arn_n] {1}
  edge from parent node[above left] {$1$}
}
child[grow=south east,emph]{ node [arn_n] {}
  child[grow=south east,emph]{node [arn_n] {}
    child[grow=south west,norm]{ node [arn_n] {2}
    edge from parent node[xshift=-.15cm,yshift=.15cm] {$0$}
    }
    child[grow=south,norm]{ node [arn_n] {}
            child{ node [arn_n] {2}
            edge from parent node[left] {$2$}
            }
            child{ node [arn_n] {3}
            edge from parent node[right] {$1$}
            }
        edge from parent node[xshift=.15cm] {$1$}
		}
    child[grow=south east,emph]{ node [arn_n] {4}
  edge from parent node[above right] {$4$}
    }
  edge from parent node[above right] {$0$}
  }
  edge from parent node[above right] {$0$}
};
\node at (-1,-4) {admissible};
\end{tikzpicture}
\qquad\qquad\qquad\qquad\qquad
\begin{tikzpicture}[-,>=stealth',level/.style={sibling distance = 5cm/#1,level distance = 1.5cm},scale=0.7] 
\node [arn_n] {}
child{ node [arn_n] {1}
  edge from parent node[above left] {$1$}
}
child[grow=south east,emph]{ node [arn_n] {}
  child[grow=south east,emph]{node [arn_n] {}
    child[grow=south west,norm]{ node [arn_n] {2}
    edge from parent node[xshift=-.15cm,yshift=.15cm] {$1$}
    }
    child[grow=south,norm]{ node [arn_n] {}
            child{ node [arn_n] {2}
            edge from parent node[left] {$2$}
            }
            child{ node [arn_n] {3}
            edge from parent node[right] {$1$}
            }
        edge from parent node[xshift=.15cm] {$2$}
		}
    child[grow=south east,emph]{ node [arn_n] {4}
  edge from parent node[above right] {$3$}
    }
  edge from parent node[above right] {$1$}
  }
  edge from parent node[above right] {$1$}
};
\node at (-2,-4) {non-admissible};
\end{tikzpicture}
\]
There are several reasons why the diagram on the right is non-admissible: it fails condition (3) because $\ell(E_1)=3\neq 4=\cpc(E_1)$, it fails condition (5) because $r=3$ and $\ell(E_3)$ is odd, and it fails condition (1) because of the existence of edges $F,F'$ with $\ell(F)=2$, $\ell(F')=1$ and $F\leq F'$.

The next result establishes the relationship between labelings and local Euler obstructions.

\begin{theorem}\label{thm:count-trees}
 The local Euler obstruction $e_{y,w}$ is equal to the number of admissible labelings of $A(w/y)$.
\end{theorem}
 
For the Grassmannians in type A and D with their corresponding Schubert stratifications, we have $e_{y,w}>0$ whenever $y\leq w$, due to the fact that the Euler obstructions can be calculated as the value at $1$ of Kazhdan--Lusztig polynomials (see \cite{boe-fu}*{Remark~6.1B}, \cite{mihalcea-singh}*{Theorem~10.4}). By contrast, for $LG(n,2n)$ the relation with Kazhdan--Lusztig polynomials is more subtle, due to the presence of reducible characteristic cycles for the intersection cohomology sheaves of the strata. In particular, the Euler obstructions can often be $0$, and it is interesting to study when this occurs. The next result identifies the main source of vanishing for Euler obstructions (see Theorem~\ref{thm:vanishing} for the complete characterization).

\begin{theorem}\label{thm:some-vanishing}
 Suppose that $\pth(y)$ \defi{lies strictly below} $\pth(w)$, that is, $y\leq w$ and $(0,0)$ is the only point that $\pth(y)$ and $\pth(w)$ have in common. We have that $e_{y,w}=0$ if and only if the following conditions hold:
\begin{itemize}
    \item the number of distinguished edges in $A(w/y)$ is odd,
    \item $E_1$ is a terminal edge and $\cpc(E_1)$ is odd,
    \item for all $i\geq 1$, the unique node incident to $E_{2i}$ and $E_{2i+1}$ is not incident to any other edge.
\end{itemize}
\end{theorem}

If we consider the diagram $A(w/y)$ in Example~\ref{ex:A-w-y} and we change the capacity of $E_1$ to an odd number, then the resulting diagram satisfies all the conditions in Theorem~\ref{thm:some-vanishing}, and in particular it corresponds to a vanishing Euler obstruction (see also Example~\ref{ex:w-y-intersect}).

As a final application, we consider the space of $n\times n$ symmetric matrices, equipped with the rank stratification, and we define $e_{i,j}$ to be the corresponding local Euler obstructions. It is shown in \cite{zhang}*{Theorem~6.6} and~\cite{lor-rai}*{Corollary~5.3} that
\[
e_{i,j} = \begin{cases}
 0 & \text{ if }n-i\text{ is even and }n-j\text{ is odd}; \\
 \displaystyle{\lfloor\frac{n-i}{2}\rfloor \choose \lfloor\frac{j-i}{2}\rfloor} & \text{otherwise}.
 \end{cases}
\]
Using the fact that the space of symmetric matrices arises naturally as the big opposite cell in $LG(n,2n)$, we explain in the last section of the paper how the formula above is a direct consequence of Theorem~\ref{thm:count-trees}.
 
\bigskip

\noindent{\bf Organization.} In Section~\ref{sec:recursion-eyw} we explain how the work of Boe and Fu leads to a positive recursion for the local Euler obstructions, and in particular it implies their positivity properties. In Section~\ref{sec:trees} we explain the combinatorics of decorated trees and admissible labelings, and prove Theorem~\ref{thm:count-trees}. We then apply this theorem in Section~\ref{sec:vanishing} to characterize all the pairs $(y,w)$ for which $e_{y,w}=0$, and we conclude in Section~\ref{sec:symm-mat} with the derivation of the local Euler obstructions for symmetric matrices.
 
\section{The recursion for local Euler obstructions}\label{sec:recursion-eyw}

The goal of this section is to explain how the results of \cite{boe-fu} lead to a proof of Theorem~\ref{thm:euler-recursion}. We begin by recalling basic facts about the Lagrangian Grassmannian \(LG(n,2n)\) and its associated Schubert stratification. We fix the symplectic vector space $(\bb{C}^{2n},\omega)$, where
\[
\omega(\vec{e}_i,\vec{e}_j) = \begin{cases}
1 & i+j=2n+1,\ i\leq n, \\
-1 & i+j=2n+1,\ i\geq n+1,\\
0 & \text{otherwise},
\end{cases}
\]
and $\{\vec{e}_i\}_i$ denotes the standard basis in $\bb{C}^{2n}$. A subspace $L\subseteq \bb{C}^{2n}$ is \defi{isotropic} if $\omega(f,g)=0$ for all $f,g\in L$. We say that $L$ is \defi{Lagrangian} if $L$ is isotropic and is maximal with respect to inclusion. $X=LG(n,2n)$ is the parameter space for Lagrangian subspaces in \(LG(n,2n)\).

The \defi{Schubert stratification} of $X$ is naturally indexed by \defi{symmetric partitions} $\lambda=(\lambda_1,\cdots,\lambda_n)$ with
\[ n\geq \lambda_1 \geq \cdots \geq \lambda_n\geq 0.\]
Here, symmetric means that $\lambda=\lambda'$, where $\lambda'$ is the \defi{conjugate partition} given by
\[ \lambda'_j = |\{i : \lambda_i\geq j\}|.\]
An example of such a partition is $\lambda = (7,6,6,3,3,3,1)$. If we write
\[ \bb{C}^s = \op{Span}(\vec{e}_1,\cdots,\vec{e}_s)\]
then the Schubert cell corresponding to $\lambda$ is given by
\begin{equation}\label{eq:def-X-lam}
X_{\lambda} = \left\{ L \in LG(n,2n) \ \middle| \ \operatorname{dim}\left(L \cap \bb{C}^{\lambda_{n+1-i} + i}\right) = i \text{ for } i \in \{1,\ldots, n\}\right\}.
\end{equation}
The closure of these cells are the corresponding Schubert varieties $\ol{X}_{\lambda}$, which can be described by replacing the equality above with \(\geq\). 

It will be useful to use a different parametrization of the Schubert cells, using words $w$ of length $n$ in the alphabet $\{\alpha,\beta\}$ (see also \cite{boe}*{Section~3}, \cite{boe-fu}*{p457}). The relation between a word $w$ and $\pth(w)$ was described in the introduction, and passing from a symmetric partition to the path of the corresponding word is illustrated best through an example. 
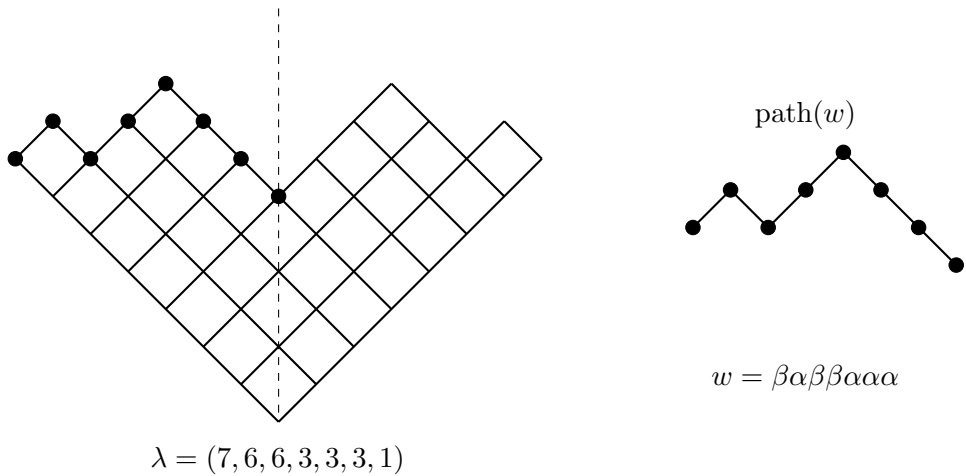
\[
\begin{array}{c}
\begin{tikzpicture}[scale=0.5]
\draw [thick] (0,0)--(7,-7);
\draw [thick] (1,1)--(8,-6);
\draw [thick] (3,1)--(9,-5);
\draw [thick] (4,2)--(10,-4);
\draw [thick] (8,0)--(11,-3);
\draw [thick] (9,1)--(12,-2);
\draw [thick] (10,2)--(13,-1);
\draw [thick] (13,1)--(14,0);
\draw [thick] (0,0)--(1,1);
\draw [thick] (1,-1)--(4,2);
\draw [thick] (2,-2)--(5,1);
\draw [thick] (3,-3)--(6,0);
\draw [thick] (4,-4)--(10,2);
\draw [thick] (5,-5)--(11,1);
\draw [thick] (6,-6)--(13,1);
\draw [thick] (7,-7)--(14,0);
\draw [dashed] (7,4)--(7,-7); 
\node at (7,-8) {$\lambda = (7,6,6,3,3,3,1)$};
\node[arn_n] at (0,0) {};
\node[arn_n] at (1,1) {};
\node[arn_n] at (2,0) {};
\node[arn_n] at (3,1) {};
\node[arn_n] at (4,2) {};
\node[arn_n] at (5,1) {};
\node[arn_n] at (6,0) {};
\node[arn_n] at (7,-1) {};
\end{tikzpicture}
\end{array}
\qquad\qquad
\begin{array}{c}
\begin{tikzpicture}[scale=0.5]
\draw [thick] (0,0)--(1,1)--(2,0)--(4,2)--(7,-1);
\node[arn_n] at (0,0) {};
\node[arn_n] at (1,1) {};
\node[arn_n] at (2,0) {};
\node[arn_n] at (3,1) {};
\node[arn_n] at (4,2) {};
\node[arn_n] at (5,1) {};
\node[arn_n] at (6,0) {};
\node[arn_n] at (7,-1) {};
\node at (3,-4) {$w = \beta\alpha\beta\beta\alpha\alpha\alpha$};
\node at (3,3) {$\pth(w)$};
\end{tikzpicture}
\end{array}
\]
To construct $\pth(w)$ for the word $w$ corresponding to $y$, we imagine the Young diagram of the symmetric partition $\lambda$ embedded into an $n\times n$ square, which we picture rotated by a $45^{\circ}$ angle. The (now) vertical diagonal of the square becomes an axis of symmetry for (the Young diagram of) $\lambda$. Starting at the left corner of the rectangle, we follow the boundary of the rectangle until we reach $\lambda$, and then follow $\lambda$, stopping when the axis of symmetry is reached. Our convention is that the left corner of the square has coordinates $(0,0)$, which is the starting point of every $\pth(w)$.

From now on we write $X_w$ for $X_{\lambda}$, where $w$ is the word corresponding to $\lambda$. The \defi{Bruhat order}, given by $y\leq w$ if and only if ($|y|=|w|$ and) $X_y\subseteq\ol{X}_w$, can be rephrased by the condition that $\pth(y)$ lies on or below $\pth(w)$. In particular, this is compatible with concatenation of words:
\[ y'\leq w'\text{ and }y''\leq w'' \Rightarrow y'y''\leq w'w''.\]
Thinking of $LG(n,2n)$ as a homogeneous space for the symplectic group $Sp(2n)$ preserving the form $\omega$, we get a natural action of the associated Weyl group $C_n$. Recall that $C_n$ is the \defi{hyperoctahedral} group of signed permutations, and it is generated by simple reflections $s_1,\cdots,s_n$. The (right) action of the simple reflections on the Schubert cells, or the corresponding words, is given as follows:
\begin{itemize}
    \item If \(i < n\) then \(ws_i\) is the word given by swapping the \(i\)th and \((i+1)\)st symbols of \(w\).
    \item \(ws_n\) is the word given by changing the last symbol in \(w\) from $\sigma$ to $\tau$ where $\{\sigma,\tau\}=\{\alpha,\beta\}$.
\end{itemize}
In the \(2\)-plane, we think of the vertical line over \((i,0)\) as corresponding to \(s_i\) as the action by this element will either interchange a local maximum of \(\pth(w)\) with a local minimum where \(\pth(w)\) intersects this line (and vice versa) or leave \(\pth(w)\) unchanged if it has no local extremum along this line. A local minimum in \(\pth(w)\) will be called a \defi{trough}, which corresponds to a subword \(\alpha \beta\) occurring in \(w\). Given two words \(y \leq w\), with a trough of \(\pth(w)\) occurring along the line corresponding to \(s_i\) we define the \defi{capacity} of \(\pth(w)\) over \(\pth(y)\) at \(s_i\), denoted \(\cpc_{s_i} (y,w)\), to be half the vertical distance from \(\pth(y)\) up to \(\pth(w)\) along the line at \(s_i\). If we write $y=y'y''$ and $w=w'w''$ with $|y'|=|w'|=i$ then
\[\cpc_{s_i} (y,w) = |y'|_{\alpha}-|w'|_{\alpha}.\]

To prove Theorem~\ref{thm:euler-recursion}, we will need the following two lemmas which are immediate consequences of \cite{boe-fu}*{Lemmas~6.2C,~6.2D}.

\begin{lemma}\label{lemma:bf-single-recur}
If \(y < w\) and if \(s\) is a simple reflection such that \(w \nless ws\) then \(e_{y,w} = e_{ys,w}\).
\end{lemma}

\begin{lemma}\label{lemma:bf-double-recur}
Suppose that \(y < w\), that \(s = s_i\) is a simple reflection such that \begin{equation}\label{eq:yys-wws}
    y < ys\text{ and }w < ws,
\end{equation}
and let \(c = \cpc_s (y,w)\) be the capacity corresponding to \(s\). There exist words \(\overline{y}, \overline{w}\) of smaller length such that
\[e_{y,w} = e_{ys,w} + (-1)^r e_{\overline{y}, \overline{w}},
\text{ where }
r = \begin{cases}
c & \text{if }i=n; \\
0 & \text{otherwise}.\\
\end{cases}\]
\begin{itemize}
    \item If \(1 \leq i < n\) then \eqref{eq:yys-wws} implies \(w = w' \alpha \beta w''\) and \(y = y' \alpha \beta y''\) where \(|w'| = |y'| = i-1\). In this case, one can take \(\overline{w} = w' w''\) and \(\overline{y} = y' y''\).
    \item If \(s = s_n\), then the conditions above imply \(w = w' \alpha\) and \(y = y' \alpha\) where \(|w'| = |y'| = n-1\). In this case, one can take \(\overline{w} = w'\) and \(\overline{y} = y'\).
\end{itemize}
\end{lemma}

The results above allow us to establish the recurrence relations stated in Theorem~\ref{thm:euler-recursion}, which combined together will imply the non-negativity of \(e_{y,w}\).

\begin{proof}[Proof of Theorem~\ref{thm:euler-recursion}(1)]
 Since $X_w$ is smooth, it is contained in the regular locus of $\ol{X}_w$, which implies $e_{w,w}=1$ \cite{macpherson}*{Section~3}. If $y\not\leq w$ then $X_y$ is disjoint from $\ol{X}_w$, hence $e_{y,w}=0$.
\end{proof}

\begin{proof}[Proof of Theorem~\ref{thm:euler-recursion}(2)]
If we consider the action of the simple reflection $s_n$ then we get
\[ys_n = (y' \alpha) s_n = y' \beta,\text{ and }ws_n = (w' \beta) s_n = w' \alpha < w.\]
If we apply Lemma~\ref{lemma:bf-single-recur} with $s=s_n$ then the desired conclusion follows.
\end{proof}

\begin{proof}[Proof of Theorem~\ref{thm:euler-recursion}(3)]
Consider the action of the simple reflection $s_i$, where $i=|y'|+1$. We have
\[ys_i = (y' \alpha \beta y'') s_i = y' \beta \alpha y'',\text{ and }ws_i = (w' \sigma \tau w'') s_i = w' \tau \sigma w'' \leq w,\]
where the last inequality follows from the fact that $\tau \sigma \leq \sigma \tau$, which holds since $\sigma\tau\neq\alpha\beta$. We may therefore conclude again by applying Lemma~\ref{lemma:bf-single-recur}. 
\end{proof}

\begin{proof}[Proof of Theorem~\ref{thm:euler-recursion}(4)]
If we let \(s=s_i\), where \(i = |y'| + 1 < n\), then we have 
\[y = y' \alpha \beta y'' < y' \beta \alpha y'' = ys\text{ and }w = w' \alpha \beta w'' < w' \beta \alpha w'' = ws.\]
The conclusion the follows from Lemma~\ref{lemma:bf-double-recur}, noting that $r=0$ since $i<n$.
\end{proof}

\begin{proof}[Proof of Theorem~\ref{thm:euler-recursion}(5--6)]
We write \(w = w' \sigma \alpha\), where \(\sigma \in \{\alpha, \beta\}\) and let
\[c = \cpc_{s_n}(y,w) = |y|_{\alpha}-|w|_{\alpha}.\]
Using the fact that
\[y = y' \alpha \alpha < y' \alpha \beta = ys_n\text{ and }w = w' \sigma \alpha < w' \sigma \beta = ws_n,\]
we can apply Lemma~\ref{lemma:bf-double-recur} with $s=s_n$ to conclude that 
\begin{equation}\label{eq:eyw=1st-approx}
e_{y,w} = e_{y' \alpha \beta, w} + (-1)^c e_{y' \alpha, w' \sigma} \overset{(3)}{=}  e_{y' \beta \alpha, w} + (-1)^c e_{y' \alpha, w' \sigma}.
\end{equation}
Repeating the calculation above with $y$ replaced by $y'\beta\alpha$, and using the fact that
\[\cpc_{s_n}(y'\beta\alpha,w) = |y'\beta\alpha|_{\alpha}-|w|_{\alpha} = c-1,\]
we find that
\[e_{y' \beta \alpha, w} = e_{y' \beta \beta,w} + (-1)^{c-1} e_{y' \beta, w' \sigma},\]
which combined with \eqref{eq:eyw=1st-approx} yields
\begin{equation}\label{eq:eyw=2nd-approx}
e_{y,w} = e_{y' \beta \beta,w} + (-1)^c e_{y' \alpha, w' \sigma} + (-1)^{c-1} e_{y' \beta, w' \sigma}.
\end{equation}

To conclude, we analyze separately the two choices for $\sigma$. If $\sigma=\beta$ then the previously established recursion (2) yields
\[e_{y' \alpha, w' \sigma} = e_{y' \beta, w' \sigma},\]
so we get a cancellation in \eqref{eq:eyw=2nd-approx} that proves $e_{y,w}=e_{y'\beta\beta}$, which establishes recursion (5) in our theorem.

Suppose now that $\sigma=\alpha$, observe that $|y'\alpha|=|w'\alpha|=n-1$, and apply Lemma~\ref{lemma:bf-double-recur} to the action of the simple reflection $s_{n-1}$ (in the corresponding smaller Weyl group). We have
\[ \cpc_{s_{n-1}}(y'\alpha,w'\alpha) = c,\]
and therefore
\[e_{y' \alpha, w' \alpha} = e_{y'\beta, w' \alpha} + (-1)^c e_{y',w'}.\]
Multiplying the equality above by $(-1)^c$ and combining it with \eqref{eq:eyw=2nd-approx} for $\sigma=\alpha$, we obtain
\[e_{y,w} = e_{y' \beta \beta,w} + (-1)^{2c}e_{y',w'} + (-1)^c e_{y' \beta, w' \alpha} + (-1)^{c-1} e_{y' \beta, w' \alpha} = e_{y' \beta \beta,w} + e_{y',w'},\]
proving recursion (6).
\end{proof}

To finish the proof of Theorem~\ref{thm:euler-recursion}, we have to explain why the recursive relations (1--6) completely determine the local Euler obstructions, which we do next.

\begin{proof}[Conclusion of the proof of Theorem~\ref{thm:euler-recursion}]
We prove that the relations (1--6) determine $e_{y,w}$, by induction on the pair $(|y|,y)$ with the order given by
\[ (|y|,y)\leq(|\tilde{y}|,\tilde{y}) \Longleftrightarrow (|y|<|\tilde{y}|) \text{ or }(|y|=|\tilde{y}|\text{ and }y\geq\tilde{y}).\]
Based on relation (1), we may assume that $y<w$. If $y$ contains the subword $\alpha\beta$, then we can either apply relation (3) or (4) and induction to compute $e_{y,w}$, using the fact that
\[ y'\beta\alpha y''>y'\alpha\beta y''\text{ and }|y'y''|<|y|.\]
We can therefore assume that $y=\beta\cdots\beta\alpha\cdots\alpha$, and since $y<w$, $y$ must contain at least one $\alpha$. If $w$ ends in $\beta$ then we apply relation (2) and induction, using the fact that $y'\beta>y$.

We may therefore further assume that $w$ ends in $\alpha$, and since $y<w$, we get that the last two letters in $y$ are $\alpha\alpha$. We can then compute $e_{y,w}$ by applying either relation (5) or (6). Noting that $y'\beta\beta>y$ and $|y'|<|y|$, we can apply induction again to conclude our proof.
\end{proof}

\section{Trees and admissible labelings}\label{sec:trees}

The goal of this section is to give a combinatorial interpretation of the local Euler obstructions for $LG(n,2n)$ as the number of admissible labelings of some tree diagrams associated with a pair of words $(y,w)$ (Theorem~\ref{thm:count-trees}). The combinatorics that we employ is familiar in Kazhdan--Lusztig theory (see \cite{las-sch}, \cite{boe}), and it was used in \cite{boe-fu} to provide a description of the local Euler obstructions as a \emph{signed} count of diagram labelings. Our contribution is to find an appropriate modification of the combinatorial constructions in order to obtain a \emph{positive} count of diagrams that elucidates the non-negativity of the local Euler obstructions. In particular, this will allow us in Section~\ref{sec:vanishing} to completely characterize the pairs $(y,w)$ for which $e_{y,w}=0$.

We let $Z$ denote the center of the \emph{cycle monoid} of Lascoux and Sch\"utzenberger \cite{las-sch}*{Section~4}, which is the smallest set of words in $\alpha$ and $\beta$ satisfying the following properties (see also \cite{boe}*{(3.7)}):
\begin{itemize}
    \item The empty word $\emptyset$ is in $Z$.
    \item If $z\in Z$ then $\alpha z \beta$ is in $Z$.
    \item $Z$ is closed under concatenation (it is a submonoid of the free group on $\alpha,\beta$).
\end{itemize}
Every word $z\in Z$ can be encoded using a rooted tree $A(z)$ constructed recursively as follows:
\begin{itemize}
    \item For the empty word, $A(\emptyset) = \bullet$ consists of only the root and no edges.
    \item The tree $A(\alpha z\beta)$ is obtained by introducing a new root, and joining it to the root of $A(z)$ by an edge.
    \item If $z_1,z_2\in Z$ then $A(z_1z_2)$ is obtained by glueing $A(z_1)$ and $A(z_2)$ at their root. In this case we will always draw the edges coming from $A(z_1)$ on the left, and those coming from $A(z_2)$ on the right.
\end{itemize}

\begin{example}\label{ex:a-z-word}
 The word $z=\alpha\beta\alpha\alpha\beta\alpha\beta\beta$ corresponds to the tree
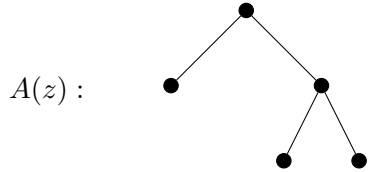
\[
A(z):\qquad
\begin{array}{c}
\begin{tikzpicture}[-,>=stealth',level/.style={sibling distance = 2cm/#1,level distance = 1cm}] 
\node [arn_n] {}
    child{ node [arn_n] {} 
    }
    child{ node [arn_n] {}
            child{ node [arn_n] {} 
            }
            child{ node [arn_n] {}
            }
		}
; 
\end{tikzpicture}
\end{array}
\]
If we consider a tubular neighborhood of $A(z)$, and travel around its boundary starting at the root, first moving along the leftmost edge, and always staying to the right in the direction of travel, then we can recover $z$ by writing the label $\alpha$ whenever we move downward along an edge, and writing $\beta$ when we move upward (see \cite{las-sch}*{Exemple~6.2}):
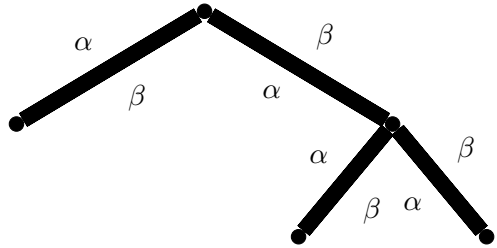
\[
\begin{tikzpicture}[-,>=stealth',level/.style={sibling distance = 5cm/#1,level distance = 1.5cm}]
\pgfsetlinewidth{15*\the\pgflinewidth}
\node [arn_n] {}
    child{ node [arn_n] {} 
    edge from parent node[above left] {$\alpha$}
    edge from parent node[below right] {$\beta$}
    }
    child{ node [arn_n] {} 
            child{ node [arn_n] {}
                edge from parent node[above left] {$\alpha$}
                edge from parent node[below right] {$\beta$}
            }
            child{ node [arn_n] {}
                edge from parent node[below left] {$\alpha$}
                edge from parent node[above right] {$\beta$}
            }
    edge from parent node[below left] {$\alpha$}
    edge from parent node[above right] {$\beta$}
    }
; 
\end{tikzpicture}
\]
One can see from the description above that words $z\in Z$ can be characterized by the condition that $\pth(z)$ starts at $(0,0)$ and ends at $(|z|,0)$, and has the property that no point on $\pth(z)$ lies above the $x$-axis.
\end{example}

We next consider $\tilde{Z}$ denote the set of words obtained by concatenating words in $Z$ and words consisting only of the letter $\beta$. A typical word $\tilde{z}\in \tilde{Z}$ has the form
\begin{equation}\label{eq:tilde-z}
\tilde{z} = z_s\beta\cdots\beta z_{s-1}\beta\cdots\beta z_{s-2} \cdots z_1 \beta\cdots\beta z_0,
\end{equation}
where some $z_i$ may be empty, as well as some of the sequences of $\beta$'s. We associate to $\tilde{z}$ the rooted tree
\[ A(\tilde{z}) := A(z_s z_{s-1}\cdots z_1 z_0).\]
Unlike for words in $Z$, it is no longer possible to recover $\tilde{z}$ from its tree. The reader can check that
\[\tilde{z} = \beta \alpha\beta \beta\beta \alpha\alpha\beta\alpha\beta\beta \beta\]
has the property that $A(\tilde{z}) = A(z)$, where $z$ is as in Example~\ref{ex:a-z-word}.

If $w$ is now an arbitrary word in $\alpha$ and $\beta$, then it can be expressed uniquely as
\begin{equation}\label{eq:w-standard-form}
     w = z_r \alpha z_{r-1} \alpha \cdots \alpha z_1 \alpha z_0,\text{ with }z_i\in Z\text{ for }i=0,\cdots,r-1,\text{ and }z_r\in\tilde{Z},
\end{equation}
where some of the $z_i$ may be empty. To any such word $w$ we associate a rooted tree $A(w)$, with 
\begin{itemize}
    \item \defi{Distinguished nodes} $V_0,\cdots,V_r$, where $V_r$ is the root of the tree $A(w)$.
    \item \defi{Distinguished edges} $E_1,\cdots,E_r$, where $E_i$ joins $V_{i-1}$ to $V_i$, and corresponds to the letter $\alpha$ in the expression \eqref{eq:w-standard-form} which is located immediately to the left of $z_{i-1}$.
    \item A subtree $A(z_i)$ attached at node $V_i$ (so that $V_i$ is the root of $A(z_i)$) for each $i=1,\cdots,r$. We call the edges occurring in the subtrees $A(z_i)$ \defi{regular edges}, and the nodes different from $V_i$ \defi{regular nodes}.
\end{itemize}

\begin{example}\label{ex:a-w-word}
The word $w=\beta\beta\alpha\beta\alpha\alpha\alpha\beta\alpha\alpha\beta\alpha\beta\beta\alpha$ corresponds to the following tree, where we represented the distinguished edges $E_1,E_2,E_3$ using dashed lines, and we highlighted the nodes $V_0,\cdots,V_3$ where the trees corresponding to the subwords $z_i$ get attached (note that $z_0$ and $z_2$ are empty, while the word $z_1$ and the corresponding subtree $A(z_1)$ of $A(w)$ are as in Example~\ref{ex:a-z-word}).
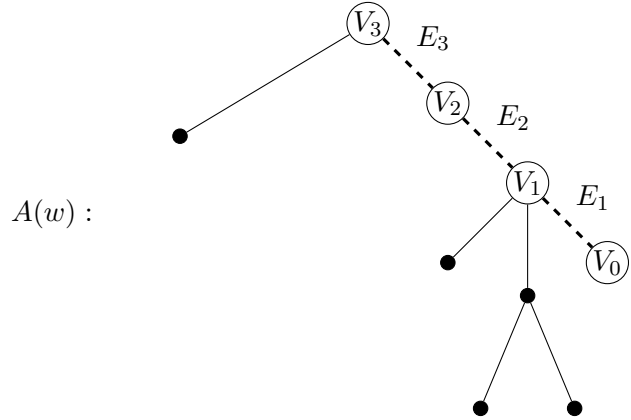
\[
A(w):\qquad
\begin{array}{c}
\begin{tikzpicture}[-,>=stealth',level/.style={sibling distance = 5cm/#1,level distance = 1.5cm}] 
\node [arn_v] {$V_3$}
child{ node [arn_n] {}
}
child[grow=south east,emph]{ node [arn_v] {$V_2$}
  child[grow=south east,emph]{node [arn_v] {$V_1$} 
    child[grow=south west,norm]{ node [arn_n] {} 
    }
    child[grow=south,norm]{ node [arn_n] {}
            child{ node [arn_n] {} 
            }
            child{ node [arn_n] {}
            }
		}
    child[grow=south east,emph]{ node [arn_v] {$V_0$}
    edge from parent node[above right] {$E_1$}
    }
  edge from parent node[above right] {$E_2$}
  }
  edge from parent node[above right] {$E_3$}
}
; 
\end{tikzpicture}
\end{array}
\]
\end{example}

A \defi{leaf} of $A(w)$ is a node which is incident to a unique edge and it is different from the root of $A(w)$. The unique edge incident to a leaf is called a \defi{terminal edge}. In Example~\ref{ex:a-w-word} there are five terminal edges: four of them are regular edges, and one is distinguished (namely~$E_1$). In general $E_1$ will be a terminal edge if and only if the word $z_0$ in \eqref{eq:w-standard-form} is empty. The distinguished edges $E_i$ with $i\geq 2$ will never be terminal. It will be important to notice that the regular terminal edges in $A(w)$ are in bijection with subwords $\alpha\beta$ of $w$ (or equivalently, with \defi{troughs} of $\pth(w)$).

There is a natural partial order on the edges of $A(w)$, where $F\leq F'$ if $F$ is contained in the unique path joining $F'$ to the root of $A(w)$. We have for instance that the distinguished edges are linearly ordered ($E_r\leq\cdots\leq E_1$), and that terminal edges are pairwise incomparable. In Example~\ref{ex:a-w-word}, we have $E_2\leq F$ for every $F$ in the subtree $A(z_1)$ with root $V_1$, but $E_1$ is incomparable to every such $F$.

We next consider a pair $(y,w)$ of words of the same size. We define a \defi{capacity} (relative to $y$) for every terminal edge of $A(w)$ as follows:
\begin{itemize}
    \item If $F$ is a regular terminal edge, corresponding to a trough $\alpha\beta$, then we write $w=w'\alpha\beta w''$ and $y=y'\sigma\tau y''$ with $|y'|=|w'|$ and $\sigma,\tau\in\{\alpha,\beta\}$, and let (see also \cite{boe}*{Example~3.9})
    \begin{equation}\label{eq:def-cap-regular}
    \cpc(F) = \cpc_y(F) := |y'\sigma|_{\alpha} - |w'\alpha|_{\alpha}.
    \end{equation}
    \item If $E_1$ is a terminal edge then we let
    \begin{equation}\label{eq:def-cap-distinguished}
    \cpc(E_1) = \cpc_y(E_1) := |y|_{\alpha}-|w|_{\alpha}.
    \end{equation}
\end{itemize}
We let $A(w/y)$ denote the tree $A(w)$ together with the additional data of the capacities of its terminal edges. We will refer to $A(w/y)$ as a \defi{diagram} or \defi{decorated tree}. 

When picturing $A(w/y)$ we will often omit the labels for nodes and edges unless we need to specifically refer to them, and we will indicate the capacities of the terminal edges by placing a number representing the capacity at the leaf incident to each terminal edge. We now encourage the reader to revisit Example~\ref{ex:A-w-y} in the Introduction, where the word $w$ is the same as the one in Example~\ref{ex:a-w-word}. Before stating the main result of the section, we ask the reader to recall Definition~\ref{def:admissible-label}, and we illustrate it with a count of admissible labelings in an example.

\begin{example}\label{ex:124-labelings}
 We claim that the number of admissible labelings of the diagram $A(w/y)$ in Example~\ref{ex:A-w-y} is $124$. Note that if $F_0$ is the unique regular edge incident to the root then $\ell(F_0)\in\{0,1\}$, and either choice is compatible with the labelings of the remaining edges. It is then enough to check that the remaining part of the diagram has $62$ admissible labelings. We have $\ell(E_1)=4$ by condition (3) in Definition~\ref{def:admissible-label}, $\ell(E_2)\leq 4$ is even by (1) and (5), and $\ell(E_3)=\ell(E_2)$ by (4). We therefore need to count admissible labelings of the decorated subtree $A(z_1)$ below, where the labels are greater than or equal to $\ell(E_2)$.
 \[
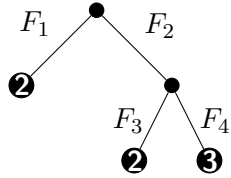
 \begin{array}{c}
\begin{tikzpicture}[-,>=stealth',level/.style={sibling distance = 2cm/#1,level distance = 1cm}] 
\node [arn_n] {}
    child{ node [arn_n] {2} 
    edge from parent node[above left] {$F_1$}
    }
    child{ node [arn_n] {}
            child{ node [arn_n] {2} 
            edge from parent node[left] {$F_3$}
            }
            child{ node [arn_n] {3}
            edge from parent node[right] {$F_4$}
            }
    edge from parent node[above right] {$F_2$}
		}
; 
\end{tikzpicture}
\end{array}
\]
It is clear that if $\ell(E_2)=4$ then no such labelings exist. If $\ell(E_2)=2$ then there are exactly two labelings:
 \[
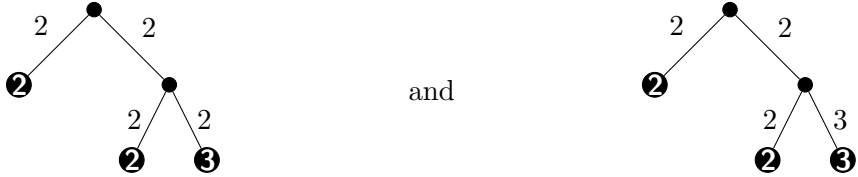
 \begin{array}{c}
\begin{tikzpicture}[-,>=stealth',level/.style={sibling distance = 2cm/#1,level distance = 1cm}] 
\node [arn_n] {}
    child{ node [arn_n] {2} 
    edge from parent node[above left] {$2$}
    }
    child{ node [arn_n] {}
            child{ node [arn_n] {2} 
            edge from parent node[left] {$2$}
            }
            child{ node [arn_n] {3}
            edge from parent node[right] {$2$}
            }
    edge from parent node[above right] {$2$}
		}
; 
\end{tikzpicture}
\end{array}
\qquad\qquad\qquad\text{and}\qquad\qquad\qquad
 \begin{array}{c}
\begin{tikzpicture}[-,>=stealth',level/.style={sibling distance = 2cm/#1,level distance = 1cm}] 
\node [arn_n] {}
    child{ node [arn_n] {2} 
    edge from parent node[above left] {$2$}
    }
    child{ node [arn_n] {}
            child{ node [arn_n] {2} 
            edge from parent node[left] {$2$}
            }
            child{ node [arn_n] {3}
            edge from parent node[right] {$3$}
            }
    edge from parent node[above right] {$2$}
		}
; 
\end{tikzpicture}
\end{array}
\]
If $\ell(E_2)=0$ then we have $\ell(F_1)\in\{0,1,2\}$ and each choice is independent of the rest of the labelings, so we need to show that the labels for $F_2,F_3,F_4$ can be chosen in $20$ ways. Indeed, if $\ell(F_2)=i$ then there are $3-i$ choices for $\ell(F_3)$ and $4-i$ choices for $\ell(F_4)$, for a total of
\[ \sum_{i=0}^2 (3-i)\cdot(4-i) = 12+6+2=20\text{ labelings}.\]
\end{example}

We next move closer toward the goal of this section, which is the proof of Theorem~\ref{thm:count-trees}, for which we establish several preliminary results.

\begin{lemma}\label{lem:rec-2}
 Suppose that $y = y' \alpha \beta y''$ and $w = w' \alpha \beta w''$, where $|y'|=|w'|$, and let $F$ denote the terminal edge of $A(w)$ corresponding to the trough $\alpha\beta$ between $w'$ and $w''$. We have a bijection between admissible labelings $\ell$ of $A(w/y)$ with $\ell(F)=\cpc(F)$ and admissible labelings of $A(w'w''/y'y'')$.
\end{lemma}

\begin{proof}
 We have that $A(w'w'')$ is obtained from $A(w)$ by removing the terminal edge $F$, so we have a natural restriction map from labelings $\ell$ of $A(w)$ to labelings $\ell'$ of $A(w'w'')$. Moreover, every labeling $\ell'$ of $A(w'w'')$ extends uniquely to a labeling $\ell$ of $A(w)$ with $\ell(F)=\cpc(F)$. We will show that this correspondence establishes the desired bijection for admissible labelings. 
 
 We let $V$ denote the node incident to $F$ which is not a leaf of $A(w)$. If $V$ is the root of $A(w)$ then $F$ is incomparable to every other edge $G$ of $A(w)$, hence its label is only subject to condition (2) in Definition~\ref{def:admissible-label}, and we get the desired bijection with labelings of $A(w'w'')$. We therefore assume that $V$ is not the root of $A(w)$, and let $F'$ denote the unique edge incident to $V$ with $F'<F$:
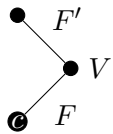
 \[
 \begin{tikzpicture}[-,>=stealth',level/.style={sibling distance = 2cm/#1,level distance = 1cm}] 
\node [arn_n] {}
    child[grow=south east,norm]{ node [arn_n,label=right:$V$] {}
            child[grow=south west,norm]{ node [arn_n] {$c$} 
            edge from parent node[below right] {$F$}
            }
        edge from parent node[above right] {$F'$}
		}
; 
\end{tikzpicture}
\]
We note that the existence of $F'$ implies that $w'$ contains at least one $\alpha$. We also note that every terminal edge $T\neq F$ of $A(w)$ is also a terminal edge in $A(w'w'')$, and it satisfies $\cpc_y(T)=\cpc_{y'y''}(T)$, making the notation $\cpc(T)$ unambiguous. We write $c=\cpc_y(F)$.
 
If $V$ is a leaf of $A(w'w'')$ then we have $\cpc_{y'y''}(F')=c$, hence every admissible labeling $\ell'$ of $A(w'w'')$ satisfies $\ell'(F')\leq c$ by Definition~\ref{def:admissible-label}(2). This is enough to conclude that the correspondence between labelings $\ell$ of $A(w/y)$ with $\ell(F)=c$ and labelings of $A(w'w''/y'y'')$ remains a bijection when restricting to admissible labelings.

If $V$ is not a leaf of $A(w'w'')$ then we will prove that there exists a terminal edge $T$ with $F'\leq T$ and $\cpc(T)\leq c$. It then follows from Definition~\ref{def:admissible-label}(1) that every admissible labeling $\ell'$ of $A(w'w'')$ satisfies $\ell'(F')\leq \ell'(T) \leq c$, and we conclude as in the previous paragraph. To find the terminal edge $T$, we note that the assumption that $V$ is not a leaf of $A(w'w'')$ implies that either $w'$ ends with $\beta$ or that $w''$ starts with $\alpha$. If $w'$ ends with $\beta$, consider the last $\alpha$ in $w'$, which is necessarily followed by a $\beta$, hence it determines a trough. If we let $T$ be the corresponding terminal edge then $\cpc(T)\leq\cpc_y(F)=c$, as desired. If $w''$ starts with $\alpha$ and contains at least one $\beta$ we consider the first such, which is necessarily preceded by $\alpha$. As in the previous case, this determines a terminal edge $T$ with $\cpc(T)\leq c$, as desired. Finally, if $w''$ contains no $\beta$ then we take $T=E_1$ to be the distinguished terminal edge, concluding the proof.
\end{proof}

\begin{lemma}\label{lem:rec-3-2}
 Suppose that $y=y'\alpha\alpha$ and $w=w'\beta\alpha$, so that $E_1$ is a distinguished terminal edge in $A(w)$. If we let $c=\cpc(E_1)$ and if $A(w)$ has at least two distinguished edges, then every admissible labeling of $A(w/y)$ satisfies $\ell(E_2)\leq c-2$.
\end{lemma}

\begin{proof} We use the notation \eqref{eq:w-standard-form}, and observe that our assumption that $A(w)$ has at least two distinguished edges, together with the fact that $w$ ends in $\beta\alpha$, implies that $z_1\neq\emptyset$. If we let $\alpha\beta$ be the last trough of $w$ then the corresponding terminal edge $F$ lies in the subtree $A(z_1)$ and moreover we have
\[ \cpc(F) < \cpc(E_1).\]
Since $F\in A(z_1)$ we have $E_2\leq F$, and parts (1) and (3) of Definition~\ref{def:admissible-label} imply that for every admissible labeling of $A(w/y)$ we have 
\[\ell(E_2)\leq\ell(F)<\ell(E_1)=c.\]
Using Definition~\ref{def:admissible-label}(5) with $i=1$, we conclude that $\ell(E_2)\leq c-2$, as desired.
\end{proof}

\begin{lemma}\label{lem:rec-3-1}
 Suppose that $y=y'\alpha\alpha$ and $w=w'\alpha\alpha$, so that $E_1$ is a distinguished terminal edge in $A(w)$. If we let $c=\cpc(E_1)$ then we have a bijection between admissible labelings of $A(w/y)$ with $\ell(E_2)=c$ and admissible labelings of $A(w'/y')$.
\end{lemma}

\begin{proof}
 Note that our hypotheses imply that with the notation \eqref{eq:w-standard-form} we have $r\geq 2$ and $z_0=z_1=\emptyset$, and that for every admissible labeling $\ell$ of $A(w/y)$ we have $\ell(E_1)=c$. If $A(w)$ has only two distinguished edges, and $\ell(E_2)=c$, then conditions (3), (4), (5) in Definition~\ref{def:admissible-label} are satisfied. Using the fact that $A(w'/y')$ is obtained from $A(w/y)$ by removing $E_1$ and $E_2$, it is clear that restricting an admissible labeling $\ell$ of $A(w/y)$ to $A(w'/y')$ provides the desired bijection.
 
 Suppose from now on that $A(w)$ has at least three distinguished edges. If $z_2=\emptyset$, or equivalently $E_3$ is a terminal edge of $A(w')$, then for every admissible labeling $\ell$ of $A(w/y)$ we have $\ell(E_3)=\ell(E_2)$. Moreover, we have $\cpc_{y'}(E_3)=c$ and therefore every admissible labeling $\ell'$ of $A(w'/y')$ satisfies $\ell'(E_3)=c$. By restricting labelings of $A(w/y)$ with $\ell(E_2)=c$ to $A(w'/y')$ we get the desired bijection.
 
 Finally, suppose that $z_2\neq\emptyset$ and let $F$ denote the terminal edge in $A(z_2)$ corresponding to the last trough of $A(w)$. We have $\cpc(F)\leq c$ and since $E_3\leq F$, every admissible labeling $\ell$ of $A(w/y)$ satisfies $\ell(E_3)\leq c$. Since $E_3$ is an odd distinguished edge in $A(w')$ but not terminal, it follows that by restricting labelings of $A(w/y)$ with $\ell(E_2)=c$ to $A(w'/y')$ we get the desired bijection.
\end{proof}

\begin{proof}[Proof of Theorem~\ref{thm:count-trees}]
 We denote by $\ell_{y,w}$ the number of admissible labelings of $A(w/y)$ and prove that they satisfy the recursions of $e_{y,w}$ from Section~\ref{sec:recursion-eyw}. We divide our analysis into several cases.
 
\noindent{\bf Case 1:} $y=w$. For every terminal edge $F$ we have $\cpc(F)=0$, hence Definition~\ref{def:admissible-label}(1) and (2) implies that $\ell(G)=0$ for every edge $G$. This is an admissible labeling, and we get $\ell_{y,w}=1$.

\noindent{\bf Case 2:} $y\not\leq w$. We will show that $\ell_{y,w}=0$. We have that some parts of $\pth(y)$ lie above $\pth(w)$, and without loss of generality, we may assume that either
\begin{itemize}
    \item $w$ contains a trough $\alpha\beta$ that lies below $\pth(y)$, in which case the corresponding regular terminal edge $F$ has $\cpc(F)<0$, and no admissible labelings exist by Definition~\ref{def:admissible-label}(2); or
    \item $E_1$ is a terminal edge (that is, $w$ ends in the letter $\alpha$) and $\cpc(E_1)<0$, in which case no admissible labelings exist by Definition~\ref{def:admissible-label}(3).
\end{itemize}

For the remaining cases we will assume that $y<w$. 

\noindent{\bf Case 3:} $y=\beta\cdots\beta\alpha$. The condition $w>y$ means that $w=\beta\cdots\beta$, thus $A(w/y)$ has no edges and $\ell_{y,w}=1$. 

\noindent{\bf Case 4:} We next show that if $y=y'\alpha$ and $w=w'\beta$ then
\begin{equation}\label{eq:lyw-rec2}
\ell_{y,w} = \ell_{y'\beta,w}.
\end{equation}
Note that since $w$ ends in $\beta$, $E_1$ is not a terminal edge. For a regular terminal edge $F$, the definition \eqref{eq:def-cap-regular} of $\cpc(F)$ does not involve the last letter in $y$ and $w$, hence $A(w/y)=A(w/y'\beta)$, which proves \eqref{eq:lyw-rec2}.

\noindent{\bf Case 5:} Suppose next that $y = y' \alpha \beta y''$ and $w = w' \sigma \tau w''$, where $|y'|=|w'|$ and $\sigma \tau \in \{\alpha \alpha, \beta \alpha, \beta \beta\}$. We show that
\begin{equation}\label{eq:lyw-rec3}
 \ell_{y,w} = \ell_{y' \beta \alpha y'',w}.
\end{equation}
It follows from \eqref{eq:def-cap-distinguished} that if $E_1$ is terminal then $\cpc(E_1)$ is not affected by permutations of the letters in $y$. Moreover, since $\sigma\tau$ is not a trough of $w$, the capacities of regular edges relative to $y$ and $y'\beta\alpha y''$ coincide. We conclude that $A(w/y)=A(w/y' \beta \alpha y'')$, which proves \eqref{eq:lyw-rec3}.

\noindent{\bf Case 6:} Suppose now that $y = y' \alpha \beta y''$ and $w = w' \alpha \beta w''$, where $|y'|=|w'|$. We show that
\begin{equation}\label{eq:lyw-rec4}
    \ell_{y,w} = \ell_{y' \beta \alpha y'' ,w} + \ell_{y' y'', w' w''}.
\end{equation}
We let $F$ denote the regular terminal edge corresponding to the trough $\alpha\beta$ between $w'$ and $w''$. If we write $\cpc_y(G)$ (resp. $\cpc_{\tilde{y}}(G)$) for the capacity of a terminal edge $G$ relative to $y$ (resp. $\tilde{y} = y' \beta \alpha y''$) then $\cpc_{\tilde{y}}(F)=\cpc_y(F)-1$ and $\cpc_{\tilde{y}}(G)=\cpc_y(G)$ for all $G\neq F$. It follows that $\ell_{y,w} - \ell_{y' \beta \alpha y'' ,w}$ is non-negative, and it counts admissible labelings of $A(w/y)$ for which $\ell(F)=\cpc(F)$. Applying Lemma~\ref{lem:rec-2} we get \eqref{eq:lyw-rec4}.

\noindent{\bf Case 7:} $y=y'\alpha\alpha$ and $w=w'\beta\alpha$. We show that 
\[ \ell_{y,w} = \ell_{y'\beta\beta,w}.\]
Let $\tilde{y} = y'\beta\beta$, let $c=\cpc_y(E_1)$, and note that $\cpc_{\tilde{y}}(E_1)=c-2$, and $\cpc_y(F)=\cpc_{\tilde{y}}(F)$ for all terminal edges $F\neq E_1$. By Definition~\ref{def:admissible-label}(3), every admissible labeling $\ell$ of $A(w/y)$ has $\ell(E_1)=c$, while every admissible labeling $\tilde{\ell}$ of $A(w/\tilde{y})$ has $\tilde{\ell}(E_1)=c-2$. Since $c$ and $c-2$ have the same parity, we get a bijection between admissible labelings $\tilde{l}$ of $A(w/\tilde{y})$ and $\ell$ of $A(w/y)$ by letting $\tilde{\ell}(G)=\ell(G)$ for all $G\neq E_1$. Indeed, if $E_1$ is the only distinguished edge of $A(w)$ then it is clear that $\tilde{\ell}$ is admissible if and only if $\ell$ is, and if $A(w)$ has at least two distinguished edges then the same conclusion follows using Lemma~\ref{lem:rec-3-2}.

\noindent{\bf Case 8:} $y=y'\alpha\alpha$ and $w=w'\alpha\alpha$. We show that 
\begin{equation}\label{eq:rec-3-1} 
\ell_{y,w} = \ell_{y'\beta\beta,w} + \ell_{y',w'}.
\end{equation}
If we let $\tilde{y} = y'\beta\beta$ and $c=\cpc_y(E_1)$, then $\cpc_{\tilde{y}}(E_1)=c-2$, and $\cpc_y(F)=\cpc_{\tilde{y}}(F)$ for all terminal edges $F\neq E_1$. It follows that every admissible labeling $\tilde{\ell}$ of $A(w/\tilde{y})$ gives rise to an admissible labeling $\ell$ of $A(w/y)$ by letting $\ell(E_1) = c = \tilde{\ell}(E_1)+2$, and $\ell(G)=\tilde{\ell}(G)$ for $G\neq E_1$. Moreover, the difference $\ell_{y,w} - \ell_{y'\beta\beta,w}$ counts admissible labelings $\ell$ of $A(w/y)$ for which changing $\ell(E_1)$ to $c-2$ results in an inadmissible labeling of $A(w/\tilde{y})$, that is, for which $\ell(E_2)=c$. Applying Lemma~\ref{lem:rec-3-1}, we get \eqref{eq:rec-3-1}.
\end{proof}

\section{Vanishing of local Euler obstructions}\label{sec:vanishing}

The goal of this section is to characterize the pairs $(y,w)$ with $y\leq w$ for which the corresponding local Euler obstructions vanish. More precisely, we use Theorem~\ref{thm:count-trees} to prove the following. 

\begin{theorem}\label{thm:vanishing}
Consider words $y,w$ with $y\leq w$.
 \begin{enumerate}
     \item Suppose that $\pth(y)$ lies strictly below $\pth(w)$. If we let $w$ as in \eqref{eq:w-standard-form} then
     \begin{equation}\label{eq:equiv-eyw=0}
      e_{y,w} = 0 \Longleftrightarrow |y|_{\alpha} - |w|_{\alpha}\text{ is odd, }r\text{ is odd, and }z_{2i}=\emptyset\text{ for all }i.
      \end{equation}
     \item Write $y=y'y''$ and $w=w'w''$ such that $|y'|_{\alpha}=|w'|_{\alpha}$, and $\pth(y'')$ lies strictly below $\pth(w'')$. We have that
     \begin{equation}\label{eq:eyw=0-iff-dble-prime=0}
     e_{y,w} = 0 \Longleftrightarrow e_{y'',w''}=0.
     \end{equation}
 \end{enumerate}
\end{theorem}

\begin{proof} We begin by proving \eqref{eq:equiv-eyw=0}, noting that the hypothesis that $\pth(y)$ lies strictly below $\pth(w)$ implies that $\cpc(T)>0$ for every terminal edge $T$ of $A(w)$ (where $\cpc(T)=\cpc_y(T)$). If $z_0\neq\emptyset$ then we get an admissible labeling $\ell$ of $A(w/y)$ by defining $\ell(G)=0$ for all $G$, which by Theorem~\ref{thm:count-trees} implies $e_{y,w}\geq 1$. We may therefore assume that $A(w)$ contains at least one distinguished edge and that $E_1$ is a terminal edge, and we write $c=|y|_{\alpha} - |w|_{\alpha}=\cpc(E_1)$. If $c$ is even then if we let $\ell(E_1)=c$ and $\ell(G)=0$ for all $G\neq E_1$, then $\ell$ is admissible and we conclude again that $e_{y,w}\geq 1$. We may thus further assume that $c$ is odd. If $r$ is even, then using the fact that all capacities are positive, we get that $\ell(E_1)=c$, $\ell(G)=1$ for $G\neq E_1$, defines an admissible labeling, so $e_{y,w}\geq 1$. We may thus further assume that $r$ is odd.

If $z_{2i}\neq 0$ for some $i$, we choose $i$ minimal with this property, and note that $i>0$ since $z_0\neq\emptyset$. If we let 
\[\ell(E_j)=\begin{cases}
c & \text{if }j=1,\\
1 & \text{if }1\leq j\leq 2i,\\
0 & \text{if }j>2i,\\
\end{cases}
\qquad\text{and for each regular edge }G\text{ let }
\ell(G)=\begin{cases}
1 & \text{if }G\in A(z_j),\ j\leq 2i,\\
0 & \text{if }G\in A(z_j),\ j>2i,\\
\end{cases}
\]
then $\ell$ is admissible, which implies $e_{y,w}\geq 1$. We may then further assume that $z_{2i}=\emptyset$ for all $i$, hence all the conditions on the right hand side of \eqref{eq:equiv-eyw=0} are satisfied. To conclude, we have to prove that $e_{y,w}=0$, or equivalently, that there is no admissible labeling of $A(w/y)$.

Suppose by contradiction that $\ell$ is an admissible labeling of $A(w/y)$. For each $i>0$, $z_{2i}=\emptyset$ implies that the only edges incident to $V_{2i}$ are the distinguished edges $E_{2i}$ and $E_{2i+1}$, hence $E_{2i}$ is the unique maximal edge which is strictly larger than $E_{2i}$, and condition (4) in Definition~\ref{def:admissible-label} implies that $\ell(E_{2i+1})=\ell(E_{2i})$. Combining this with condition (5) in the same definition, and with the fact that $\ell(E_1)=c$ is odd, we conclude that $\ell(E_j)$ is odd for all $1\leq j\leq r$. Moreover, since $r$ is odd we get also that 
\[0 = \ell(E_{r+1}) \equiv \ell(E_{r}) \ (\opmod 2),\]
which is a contradiction and concludes the proof of \eqref{eq:equiv-eyw=0}.

To prove part (2) of the theorem, notice that the hypothesis implies that $y''$ starts with $\alpha$, while $w''$ starts with $\beta$. If $|w'|_{\alpha}=0$ then $w'=y'=\beta\cdots\beta$, hence $A(w/y)=A(w''/y'')$ and \eqref{eq:eyw=0-iff-dble-prime=0} follows from Theorem~\ref{thm:count-trees}. We may therefore assume that $|w'|_{\alpha}>0$ and consider the last $\alpha$ in $w'$, which is necessarily followed in the word $w$ by $\beta$. We get a trough $\alpha\beta$ of $w$ and let $F$ denote the corresponding regular terminal edge. With the notation \eqref{eq:w-standard-form}, we have that $F\in A(z_j)$ for a unique $j$. We can write $z_j$ as the concatenation $z_j'z_j''$, with $z_j''$ a (possibly empty) subword of $w''$, $z_j'$ and $w''$ have no $\alpha$ in common, and $z_j',z_j''\in Z$ if $j<r$, while $z_j',z_j''\in\tilde{Z}$ if $j=r$. We have that $A(z_j)$ is obtained by glueing the trees $A(z_j')$ with $A(z_j'')$ at their root, and that $F\in A(z_j')$. Moreover, $A(w'')$ is the subtree of $A(w)$ consisting of the following:
\begin{itemize}
    \item distinguished edges $E_1,\cdots,E_{j}$ (incident to the distinguished nodes $V_0,\cdots,V_j$).
    \item the subtree $A(z_i)$ attached to $V_i$ for $i<j$, and the subtree $A(z_j'')$ attached to $V_j$.
\end{itemize}
Moreover, since $|y'|_{\alpha}=|w'|_{\alpha}$, the capacity of a terminal edge $T$ in $A(w'')$ relative to $y''$ is the same as that relative to $y$ when we view $T$ as an edge in $A(w)$. Moreover, since $F$ corresponds to the last trough of $w$ not in $w''$, $|y'|_{\alpha}=|w'|_{\alpha}$ and $y'\leq w'$, we must have $\cpc(F)=0$ (see also Example~\ref{ex:w-y-intersect} below, where $j=3$ and $A(z_j')$ consists precisely of the edges labeled $G$ and $F$). 

With the notation above, we can now prove \eqref{eq:eyw=0-iff-dble-prime=0}. Suppose first that $e_{y,w}\geq 1$ and let $\ell$ be an admissible labeling of $A(w/y)$. Since $\cpc(F)=0$, we get $\ell(F)=0$, hence for $i>j$ we have $E_i\leq F$, which yields $\ell(E_i)=0$. This implies that if we restrict $\ell$ to a labeling $\ell''$ of $A(w''/y'')$, then $\ell''$ still satisfies condition (5) in Definition~\ref{def:admissible-label}. The remaining conditions are easily seen to be preserved by the restriction, hence $\ell''$ is admissible and $e_{y'',w''}\geq 1$.

Conversely, suppose that $e_{y'',w''}\geq 1$, and consider an admissible labeling $\ell''$ of $A(w''/y'')$. We extend $\ell''$ to a labeling $\ell$ of $A(w/y)$ by setting $\ell(E_i)=0$ for $i>j$ and $\ell(G)=0$ for all $G\in A(z_i)$ with $i>j$, and for all $G\in A(z'_j)$. Note that $\ell$ satisfies conditions (1)--(3) in Definition~\ref{def:admissible-label}. It also satisfies condition (4), with a potential exception for the edge $E_{j+1}$ if $j+1=2i-1$ is odd: since the tree $A(z'_j)$ contains at least the edge $F>E_{j+1}$, our construction of the labeling has $\ell(E_{j+1})=0=\ell(F)$, hence (4) in fact holds. Similarly, it is clear that condition (5) in Definition~\ref{def:admissible-label} holds for $\ell$, with a potential exception when $2i-1=j$ is odd: in this case, the fact that $\ell''$ is admissible together with our convention in (5) implies that $\ell''(E_j)$ is even, hence $\ell(E_j)=\ell''(E_j)$ and $\ell(E_{j+1})=0$ have the same parity, so in fact (5) holds for $\ell$. It follows that $\ell$ is admissible, and therefore $e_{y,w}\geq 1$, concluding the proof.
\end{proof}

\begin{example}\label{ex:w-y-intersect}
 Let $w=w'w''$, where
 \[ w' = \beta\alpha\beta\beta\alpha\alpha\alpha\beta\qquad\text{ and }\qquad w''=\beta\alpha\alpha\beta\alpha\beta\beta\alpha\alpha\alpha\beta\alpha,\]
 and let $y=y'y''$, where
 \[ y' = \alpha\alpha\beta\alpha\beta\alpha\beta\beta\qquad\text{ and }\qquad y'' = \alpha\alpha\alpha\alpha\alpha\beta\alpha\beta\beta\alpha\alpha\alpha.\]
 We have $y\leq w$, $|y'|_{\alpha}=|w'|_{\alpha}=4$, and $\pth(y'')$ lies strictly below $\pth(w'')$, so the hypothesis of Theorem~\ref{thm:vanishing}(2) holds. The paths $y$ and $w$, together with the relevant capacities, are pictured as follows.
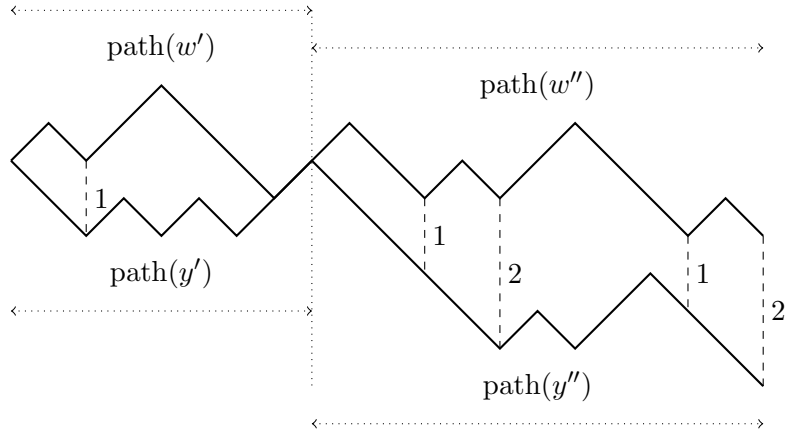
\[
\begin{array}{c}
\begin{tikzpicture}[scale=0.5]
\draw [thick] (0,0)--(1,1)--(2,0)--(3,1)--(4,2)--(5,1)--(6,0)--(7,-1)--(8,0)--(9,1)--(10,0)--(11,-1)--(12,0)--(13,-1)--(14,0)--(15,1)--(16,0)--(17,-1)--(18,-2)--(19,-1)--(20,-2);
\draw [thick] (0,0)--(1,-1)--(2,-2)--(3,-1)--(4,-2)--(5,-1)--(6,-2)--(7,-1)--(8,0)--(9,-1)--(10,-2)--(11,-3)--(12,-4)--(13,-5)--(14,-4)--(15,-5)--(16,-4)--(17,-3)--(18,-4)--(19,-5)--(20,-6);
\draw [dashed] (2,0)--(2,-2); 
\node at (2.4,-1) {$1$};
\draw [dashed] (11,-1)--(11,-3);
\node at (11.4,-2) {$1$};
\draw [dashed] (13,-1)--(13,-5);
\node at (13.4,-3) {$2$};
\draw [dashed] (18,-2)--(18,-4);
\node at (18.4,-3) {$1$};
\draw [dashed] (20,-2)--(20,-6);
\node at (20.4,-4) {$2$};
\draw [dotted] (8,4)--(8,-6);
\draw [<->,dotted] (0,-4)--(8,-4);
\node at (4,-3) {$\pth(y')$};
\draw [<->,dotted] (8,-7)--(20,-7);
\node at (14,-6) {$\pth(y'')$};
\draw [<->,dotted] (8,3)--(20,3);
\node at (14,2) {$\pth(w'')$};
\draw [<->,dotted] (0,4)--(8,4);
\node at (4,3) {$\pth(w')$};
\end{tikzpicture}
\end{array}
\]
The associated diagram $A(w/y)$ is given below, where $F,G,H$ and $E_4$ are the edges of $A(w)$ which do not belong to $A(w'')$.
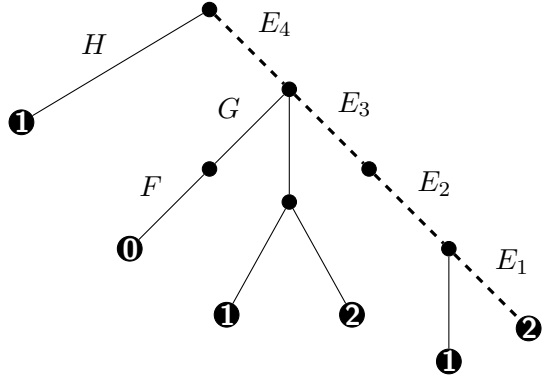
\[
\begin{tikzpicture}[-,>=stealth',level/.style={sibling distance = 5cm/#1,level distance = 1.5cm}] 
\node [arn_n] {}
child{ node [arn_n] {1}
    edge from parent node[above left] {$H$}
}
child[grow=south east,emph]{ node [arn_n] {}
    child[grow=south west,norm]{ node [arn_n] {} 
        child[grow=south west,norm]{ node [arn_n] {0} 
              edge from parent node[above left] {$F$}
        }
    edge from parent node[above left] {$G$}
    }
    child[grow=south,norm]{ node [arn_n] {}
            child{ node [arn_n] {1} 
            }
            child{ node [arn_n] {2}
            }
		}
    child[grow=south east,emph]{ node [arn_n] {}
        child[grow=south east,emph]{ node [arn_n] {}
            child[grow=south,norm]{ node [arn_n] {1}
            }
            child[grow=south east,emph]{ node [arn_n] {2}
                  edge from parent node[above right] {$E_1$}
            }
          edge from parent node[above right] {$E_2$}
        }
    edge from parent node[above right] {$E_3$}
    }
edge from parent node[above right] {$E_4$}
}
; 
\end{tikzpicture}
\]
The reader can check by an argument similar to the one in Example~\ref{ex:124-labelings} that $e_{y,w}=32$ and $e_{y'',w''}=16$. To see at least that $e_{y,w}=2e_{y'',w''}$, it suffices to observe that every admissible labeling of $A(w''/y'')$ extends in two ways to an admissible labeling of $A(w/y)$, namely by letting $\ell(F)=\ell(G)=\ell(E_4)=0$ and $\ell(H)\in\{0,1\}$. 

We can make a slight modification to $y''$, which results in a dramatic change of the corresponding Euler obstruction. If we change the last letter of $y''$ from  $\alpha$ to $\beta$, then the capacity of the distinguished terminal edge becomes $\cpc(E_1)=1$, which is odd. Part (1) of Theorem~\ref{thm:vanishing} then implies that $e_{y'',w''}=0$, and by part (2) we get that $e_{y,w}=0$ as well.
\end{example}

\begin{remark}
 A careful analysis of the proof of Theorem~\ref{thm:vanishing} and Example~\ref{ex:w-y-intersect} shows that in fact under the hypothesis in case (2) of Theorem~\ref{thm:vanishing}, $e_{y,w}$ is always an integer multiple of $e_{y'',w''}$, as follows. If we let $a_i$ (resp. $a'_j$) denote the number of admissible labelings of the subdiagram corresponding to $A(z_i)$ (resp. $A(z'_j)$), then we have
 \[e_{y,w} = a'_j\cdot a_{j+1}\cdot a_{j+2}\cdots a_r\cdot e_{y'',w''}.\]
\end{remark}

\section{Symmetric matrices}\label{sec:symm-mat}

Consider the complex vector space $U$ of $n\times n$ symmetric matrices, with its rank stratification where $U_i$ denotes the stratum of rank $i$ matrices. We let
\[ e_{i,j} = \Eu_{\ol{U}_j}(x_i) \text{ for }x_i\in U_i,\]
where $\ol{U}_j = U_0\sqcup\cdots\sqcup U_j$ is the variety of matrices of rank at most $j$.

\begin{theorem}[\cite{zhang}*{Theorem~6.6}, \cite{lor-rai}*{Corollary~5.3}]\label{thm:symmetric-matrices}
 The local Euler obstructions for the rank stratification of the space of symmetric matrices are given for $i>j$ by $e_{i,j}=0$, and for $0\leq i\leq j\leq n$ by
\[ e_{i,j} = \begin{cases}
 0 & \text{ if }n-i\text{ is even and }n-j\text{ is odd}; \\
 \displaystyle{\lfloor\frac{n-i}{2}\rfloor \choose \lfloor\frac{j-i}{2}\rfloor} & \text{otherwise}.
 \end{cases}\]
\end{theorem}

\begin{proof} Following \cite{lak-rag}*{Section~6.2.5}, we can realize $U$ as an open subset of $LG(n,2n)$ (the dense orbit relative to the action of the opposite Borel), and moreover, if we restrict the Schubert stratification to $U$ then we get a refinement of the rank stratification. More precisely, we have that
\[ X_w \cap U \subseteq U_j \Longleftrightarrow |w|_{\beta} = j.\]
In particular, the largest word $w$ with $X_w\cap U\subseteq U_j$ is given by
\begin{equation}\label{eq:w-rank-j}
w = \beta\cdots\beta\alpha\cdots\alpha = \beta^j \alpha^{n-j},
\end{equation}
and therefore we have that
\[\ol{X_w\cap U} = \ol{U}_j.\]
If we take $y=\beta^i\alpha^{n-i}$, $i\leq j$, and let $w$ as in \eqref{eq:w-rank-j}, then we then obtain
\[ e_{i,j} = e_{y,w},\]
which we then compute using Theorem~\ref{thm:count-trees}. We note that $A(w/y)$ has a simple form, namely it consists only of $n-j$ distinguished edges, and the capacity of $E_1$ is given by $\cpc(E_1)=j-i$:
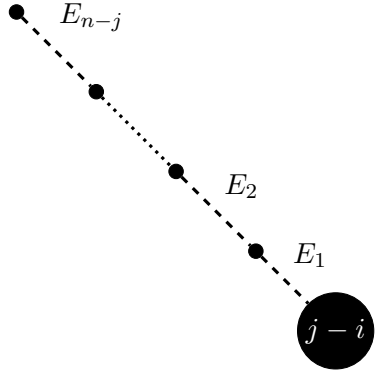
\[
\begin{tikzpicture}[-,>=stealth',level/.style={sibling distance = 5cm/#1,level distance = 1.5cm}] 
\node [arn_n] {}
child[grow=south east,emph]{ node [arn_n] {}
    child[grow=south east,dott]{ node [arn_n] {}
        child[grow=south east,emph]{ node [arn_n] {}
            child[grow=south east,emph]{ node [arn_nl] {$j-i$}
                  edge from parent node[above right] {$E_1$}
            }
          edge from parent node[above right] {$E_2$}
        }
    }
edge from parent node[above right] {$E_{n-j}$}
}
; 
\end{tikzpicture}
\]
If we write $\ell_t = \ell(E_t)$ then an admissible labeling corresponds to a sequence
\[ j-i = \ell_1 \geq \ell_2 = \ell_3 \geq \ell_4=\ell_5 \geq \cdots \]
where all $\ell_t$ have the same parity, and $\ell_{n-j}$ is even when $n-j$ is odd. This is impossible if both $(j-i)$ and $(n-j)$ are odd, or equivalently, if $(n-i)$ is even and $(n-j)$ is odd, proving that $e_{i,j}=e_{y,w}=0$ in that case. Otherwise, if we let $a_t = \lfloor\ell_{2t+1}/2\rfloor$, then the choice of the labeling $\ell$ is equivalent to the choice of an integer sequence satisfying
\[ \left\lfloor \frac{j-i}{2}\right\rfloor \geq a_1 \geq \cdots \geq a_{\lfloor(n-j)/2\rfloor} \geq 0.\]
It follows from Theorem~\ref{thm:count-trees} that
\[ e_{i,j} = e_{y,w} = \displaystyle{\lfloor\frac{j-i}{2}\rfloor + \lfloor\frac{n-j}{2}\rfloor \choose \lfloor\frac{j-i}{2}\rfloor} = \displaystyle{\lfloor\frac{n-i}{2}\rfloor \choose \lfloor\frac{j-i}{2}\rfloor},\]
where the last equality uses the fact that not both $(j-i)$ and $(n-j)$ are odd.
\end{proof}

\section*{Acknowledgements}
We are grateful to Leonardo Mihalcea for helpful suggestions and clarifications regarding our manuscript. Experiments with the computer algebra software Macaulay2 \cite{M2} have provided numerous valuable insights. Raicu acknowledges the support of the National Science Foundation Grant No.~1901886.

\end{document}